\documentclass[reqno]{amsart}
\usepackage{amsfonts, mathtools, amsmath, amsthm, amssymb, latexsym, xfrac, mathrsfs, braket}
\usepackage{framed,graphicx,xcolor,tikz,enumitem}
\definecolor{shadecolor}{gray}{0.9}

\theoremstyle{plain}
\newtheorem{theorem}{Theorem}[section]
\newtheorem{corollary}[theorem]{Corollary}
\newtheorem{lemma}[theorem]{Lemma}
\newtheorem{proposition}[theorem]{Proposition}
\theoremstyle{definition}
\newtheorem{definition}[theorem]{Definition}
\theoremstyle{remark}
\newtheorem{remark}[theorem]{Remark}

\numberwithin{equation}{section}

\newcommand{\C}{\mathbb{C}}
\newcommand{\R}{\mathbb{R}}

\newcommand{\Z}{\mathbb{Z}}

\newcommand{\N}{\mathbb{N}}

\newcommand{\tm}{\subseteq}

\newcommand{\supp}{\mathrm{supp}}

\newcommand{\id}{\mathrm{id}}
\renewcommand{\Re}{\mathrm{Re}}

\renewcommand{\d}{\mathrm{d}}
\newcommand{\dif}{\; \mathrm{d}}

\newcommand{\nrm}[1]{\left\lVert#1\right\rVert}
\newcommand{\abs}[1]{\ensuremath{\left\vert#1\right\vert}}

\newcommand{\te}{\textrm}
\newcommand{\m}{\cdot} 
\renewcommand{\tilde}[1]{\widetilde{#1}}

\renewcommand{\set}[1]{\left\{ #1 \right\}}

\renewcommand\labelenumi{(\roman{enumi})}
\renewcommand\theenumi\labelenumi
\makeindex
\title[Dunkl convolution and elliptic regularity for Dunkl operators]
{Dunkl convolution and elliptic regularity for Dunkl operators}
\author{Dominik Brennecken} 
\address{Institut f\"ur Mathematik, Universit\"at Paderborn, Warburger Str. 100, D-33098 Paderborn, Germany}
\email{bdominik@math.upb.de}


\begin{document}
\date{\today}

\begin{abstract}
We discuss in which cases the Dunkl convolution $u*_kv$ of distributions $u,v$, possibly both with non-compact support, can be defined and study its analytic properties.  We prove results on the (singular-)support of Dunkl convolutions. Based on this, we are able to prove a theorem on elliptic regularity for a certain class of Dunkl operators, called elliptic Dunkl operators. Finally, for the root system $A_{n-1}$ we consider the Riesz distributions $(R_\alpha)_{\alpha \in \C}$ and prove that their Dunkl convolution exists and that $R_\alpha*_kR_\beta = R_{\alpha+\beta}$ holds.
\end{abstract}

\maketitle
%%%%%%%%%%%%%%%%%%%%%%%%%%

\section{Introduction}
The analysis of linear differential operators has a wide range and in particular the theory of elliptic operators has a long history, see for instance \cite{FJ98,H03}. Consider a Riemannian manifold $M$ and a linear partial differential operator $D$ on $M$. If $D$ is an elliptic operator, such as the Laplace-Beltrami operator, it is well known that $D$ is hypoelliptic, which means
$$\mathrm{singsupp}(Du)=\mathrm{singsupp}\, u$$
for any distribution $u$ on $M$, where $\mathrm{singsupp}$ denotes the singular support. Furthermore, for an elliptic operator $D$ of order $m$, various regularity results about the action on Sobolev spaces are known, such as
$$Du \in H^s_{loc}(M) \te{ if and only if } u \in H^{s+m}_{loc}(M),$$
where $H^s_{loc}(M)$ is the local Sobolev space of order $s$ on $M$. The aim of this paper is to describe and prove such results for rational Dunkl operators. To point out why one might expect that this is true, we explain the connection between the analysis of Dunkl operators and radial analysis on Riemannian symmetric spaces of Euclidean type. \\
Consider a connected semisimple Lie group $G$ with finite center and maximal compact subgroup $K$. Let $\mathfrak{g}=\mathfrak{k}\oplus \mathfrak{p}$ be the associated Cartan decomposition, where $\mathfrak{g}$ and $\mathfrak{k}$ are the Lie algebras of $G$ and $K$, respectively. Then, $K$ acts on $\mathfrak{p}$ by the adjoint representation, so that for the Cartan motion group $G_0\coloneqq K\ltimes\mathfrak{p}$ we obtain an associated Riemannian symmetric space of Euclidean type
$$M=\mathfrak{p} \cong G_0/K.$$
Let $\mathfrak{a}\tm \mathfrak{p}$ be a maximal abelian subspace and consider the corresponding set of restricted roots $\Sigma\tm \mathfrak{a}$ with Weyl group $W$. Any $K$-invariant linear partial differential operator $D$ on $\mathfrak{p}$ is of the form $D=p(\partial)$ with $p \in \C[\mathfrak{p}]^K \cong \C[\mathfrak{a}]^W$, uniquely described through its restriction to $\mathfrak{a}$ by Chevalley's restriction theorem. Since similarly any $K$-invariant function $f$ on $\mathfrak{p}$ can be identified with a unique $W$-invariant function on $\mathfrak{a}$ by restriction, there exists a unique $W$-invariant differential operator $\mathrm{rad}(p(\partial))$ on $\mathfrak{a}$, called the radial part, such that for all $f \in C^\infty(\mathfrak{p})^K$
$$(p(\partial)f)|_{\mathfrak{a}}=\mathrm{rad}(p(\partial))f|_{\mathfrak{a}}.$$ Let $(m_\alpha)_{\alpha \in \Sigma}$ be the root multiplicities, i.e. the dimension of the root spaces. Choose a reduced root system $R \tm \Sigma$ and define
$$k_\alpha\coloneqq \frac{1}{4}\sum\limits_{\beta \in \R\alpha \cap \Sigma} m_\beta, \quad \alpha \in R.$$
In \cite{dJ93}, the author observed that
$$\mathrm{rad}(p(\partial)) = \mathrm{res}(p|_{\mathfrak{a}}(T)),$$
where $\mathrm{res}(p|_{\mathfrak{a}}(T))$ is the unique $W$-invariant differential operator on $\mathfrak{a}$, which coincides with the Dunkl operator $p|_{\mathfrak{a}}(T)$ on $W$-invariant functions, defined by plugging into $p|_{\mathfrak{a}}$ the commuting directional Dunkl operators
$$T_\xi f(x) \coloneqq \partial_\xi f(x) + \frac{1}{2}\sum\limits_{\alpha \in R} k_\alpha \braket{\alpha,\xi}\frac{f(x)-f(s_\alpha x)}{\braket{\alpha,x}}, \quad \xi \in \mathfrak{a},$$
where $s_\alpha$ is the reflection in the hyperplane perpendicular to $\alpha$. \\
For instance, the Laplace-Beltrami operator $\Delta_{\mathfrak{p}}$ on $\mathfrak{p}$ is $K$-invariant and the radial part can be written as
$$\mathrm{rad}(\Delta_{\mathfrak{p}})=\Delta_{\mathfrak{a}}+\sum\limits_{\alpha \in R_+} k_\alpha \frac{\partial_\alpha}{\braket{\alpha,\m}},$$
where $\Delta_{\mathfrak{a}}$ is the usual Laplacian on $\mathfrak{a}$. 
From this observation, one might expect that elliptic Dunkl operators $q(T)$, i.e. operators $q(T)$ such that the highest order term of $q$ does not vanish on $\mathfrak{a}\backslash \set{0}$, also satisfy some elliptic regularity theorems. One might expect this at least in the $W$-invariant case, but now for arbitrary parameters $k\ge 0$ and not only for such which are related to a Riemannian symmetric space of Euclidean type. This seems to be plausible, as ellipticity only depends on the highest order term, which is independent of $k$. For instance, hypoellipticity of the Dunkl Laplacian was already proven in \cite{MT04}. Important tools are the Dunkl transform and generalized translations, which define the Dunkl convolution $*_k$. Basic ideas for the results and proofs in this paper are in line with those of classical elliptic regularity such as in \cite{FJ98,H03}. However, there are two problems that need to be circumvented. Firstly, the missing (general) Leibniz rule for Dunkl operators, and second the missing knowledge about the support of generalized translations. The most important property of the Dunkl convolution we are able to prove here states that
$$\mathrm{supp}(u*_k v) \tm B_r(0)+W.\mathrm{supp} \,v,$$
for any distributions $u,v$ on $\mathfrak{a}$ such that $\supp\, u$ is contained in the closed ball $B_r(0)$ of radius $r$. This behavior of the support is based on an important result of \cite{DH19} on the support of generalized translations of $L^2$-functions. \\
The paper is organized as follows. We start in Section 2 with a brief introduction in rational Dunkl theory, in particular Dunkl transform and generalized translations. \\ 
In Section 3, we discuss several properties of the Dunkl convolution, which generalizes convolutions of $K$-invariant functions on $\mathfrak{p}\cong G_0/K$. We study the Dunkl convolution of two distributions and obtain information about the support of a convolution. This extends \cite{OS05}, where one of the distributions was required to have compact support. 
In Section 4, we introduce a generalized singular support $\mathrm{singsupp}_k u$ of a distribution $u$, which is defined as the complement of the largest open subset on which $u$ coincides with a function $f\omega$ with $f \in C^\infty(\mathfrak{a})$ and $\omega(x)=\prod_{\alpha \in R}\abs{\braket{\alpha,x}}^{k_\alpha}$. This singular support is consistent with the Dunkl setting, and we examine how that singular support behaves under convolution. 
In Section 5 we give a proof for hypoellipticity of elliptic Dunkl operators, based on the results of the previous sections. To be more precise, for an elliptic Dunkl operator $p(T)$ we prove that
$$W.\mathrm{singsupp}_k(p(T)u)=W.\mathrm{singsupp}_k\, u,$$ 
for all distributions $u$ defined on an open $W$-invariant subset of $\mathfrak{a}$. 
In Section 6 we prove the following elliptic regularity theorem for an elliptic Dunkl operator $p(T)$ of degree $m$, stating
$$p(T)u \in H_{k,loc}^s(\Omega) \te{ if and only if } u \in H_{k,loc}^{s+m}(\Omega),$$
where $H_{k,loc}^s(\Omega)$ are generalized local Sobolev spaces on some $W$-invariant open $\Omega\tm \mathfrak{a}$, as introduced in \cite{T01,MT04}, cf. Section 6.\\
Finally, as an application we will prove that the Dunkl convolution of Riesz distributions associated to the root system of type $A$, introduced and studied in \cite{R20}, exists and that they form a group under Dunkl convolution. This is in line with classical results on Riesz distributions for symmetric cones as in \cite{FK94}.

\section{The rational Dunkl setting}
For a general background on rational Dunkl theory the reader is referred to \cite{Dun89,dJ93,R03,DX14}. 
Let $(\mathfrak{a},\langle \cdot ,\cdot \rangle)$ be a finite-dimensional Euclidean space with norm $|x|\coloneqq \sqrt{\langle x,x\rangle}$ and complexification $\mathfrak{a}_\C=\C\otimes \mathfrak{a}$. Fix a reduced root system $R \subseteq \mathfrak{a}$ with associated finite reflection group $W$ and let $k:R \to \C, \, \alpha \mapsto k_\alpha$ be a $W$-invariant function, called multiplicity. Throughout the paper, $k$ will be fixed and non-negative. As usual, $W$ acts on functions $f: \mathfrak{a} \to \C$ by the assignment $w.f(x)=f(w^{-1}x)$. The rational Dunkl operator associated to $(R,k)$ into direction $\xi \in \mathfrak{a}$ acting on $f \in C^1(\mathfrak{a})$ is defined by
$$T_\xi f(x)\coloneqq T_\xi(k)f(x)\coloneqq \partial_\xi f(x) + \frac{1}{2}\sum_{\alpha \in R} k_\alpha \braket{\alpha,\xi} \frac{f(x)-f(s_\alpha x)}{\braket{\alpha,x}},$$
where $s_\alpha x= x-2\frac{\braket{x,\alpha}}{\braket{\alpha,\alpha}}\alpha$ is the reflection in the hyperplane perpendicular to $\alpha$. The Dunkl operators commute for fixed $(R,k)$, i.e. $T_\xi T_\eta=T_\eta T_\xi$ on $C^2(\Omega)$ for any $W$-invariant open $\Omega \subseteq \mathfrak{a}$, so that $p(T)$ is well-defined for any polynomial function $p$ on $\mathfrak{a}$. The Dunkl kernel
$$E\coloneqq E_k:\mathfrak{a}_\C\times\mathfrak{a}_\C \to \C$$
is the unique holomorphic function such that for any $\lambda \in \mathfrak{a}$ the function $f(x)=E(\lambda,x)$ is the unique analytic solution of the joint eigenvalue problem
$$\begin{cases}
\;T_\xi f = \braket{\lambda,\xi}f, & \te{for all }\xi \in \mathfrak{a}, \\
f(0)=1.
\end{cases}$$
$E$ is positive on $\mathfrak{a}$ and satisfies
$$E(w\lambda,wz)=E(\lambda,z), \quad E(s\lambda,z)=E(\lambda,sz) \te{ and } E(\lambda,z)=E(z,\lambda)$$
for all $\lambda,z \in \mathfrak{a}_\C$, $s \in \C$ and $w \in W$. Together with the weight function
$$\omega(x)\coloneqq \omega_k(x)\coloneqq \prod\limits_{\alpha \in R} |\langle\alpha,x\rangle|^{k_\alpha},$$
the Dunkl kernel defines the kernel of the Dunkl transform
$$\mathcal{F}_kf(\xi) \coloneqq \widehat{f}^{\,k}(\xi)\coloneqq \frac{1}{c_k}\int_{\mathfrak{a}}E(-i\xi,x)f(x)\omega(x) \; \mathrm{d} x$$
for any function $f \in L^1(\mathfrak{a},\omega)\coloneqq L^1(\mathfrak{a},\omega(x)\mathrm{d}x)$ and with the Macdonald-Metha constant
\begin{equation}\label{MacdonaldMetha}
c_k\coloneqq \int_{\mathfrak{a}}e^{-|x|^2/2}\omega(x)\; \mathrm{d}x.
\end{equation}
The Dunkl transform shares several properties with the Euclidean Fourier transform, which shows up in the case $k=0$. We summarize some important facts in the following Lemma.

\begin{lemma}\label{DunklTrafo}
\
\begin{enumerate}
\item[\rm (i)] Riemann-Lebesgue lemma: the Dunkl transform of $f \in L^1(\mathfrak{a},\omega)$ is continuous and vanishes at infinity.
\item[\rm (ii)] Plancherel theorem: $\mathcal{F}_k$ extends to a unitary map of $L^2(\mathfrak{a},\omega)$.
\item[\rm (iii)] $\mathcal{F}_k$ is injective on $L^1(\mathfrak{a},\omega)$.
\item[\rm (iv)] $\mathcal{F}_k$ is an automorphism of the Schwartz space $\mathscr{S}(\mathfrak{a})$ satisfying
$$T_\xi\mathcal{F}_k=-\mathcal{F}_km_{i\xi} \quad \te{ and } \quad \mathcal{F}_kT_\xi=m_{i\xi}\mathcal{F}_k,$$
where $m_z$ is the multiplication operator $f\mapsto \braket{z,\m}f$ for $z \in \mathfrak{a}_\C$. The inverse is given by
$$\mathcal{F}_k^{-1}f(x)=\mathcal{F}_kf(-x).$$
\end{enumerate}
\end{lemma}

Let $\mathcal{P}$ be the space of polynomial functions on $\mathfrak{a}$ and denote by $\mathcal{P}_n$ the subspace of homogeneous polynomials of degree $n \in \N_0$. There exists a unique automorphism $V=V_k$ of $\mathcal{P}$ such that
$$V1=1, \quad V\partial_\xi=T_\xi V \; \te{ and } \; V(\mathcal{P}_n) \subseteq \mathcal{P}_n,$$
called Dunkl's intertwining operator. Due to Rösler \cite{R99}, $V$ is a positivity preserving operator, i.e. $p\ge 0$ implies $Vp\ge 0$, and for all $x \in \mathfrak{a}$ there exists a unique probability measure $\mu_x^k$ with
$$Vp(x) = \int_{\mathfrak{a}} p\;\mathrm{d}\mu_x^k.$$
Trim\`eche \cite{T01} extended the operator $V$ to a topological automorphism of $C^\infty(\mathfrak{a})$, equipped with the usual locally convex topology. The generalized translation of $\varphi \in C^\infty(\mathfrak{a})$ is then defined by
$$\tau_x\varphi(y)\coloneqq \tau_x^k\varphi(y) \coloneqq V^xV^y(V^{-1}\varphi)(x+y), \quad x,y \in \mathfrak{a},$$
where the superscript $x$ and $y$ denote the relevant variables. For our purpose, the following properties of the generalized translation are of relevance.

\newpage

\begin{lemma}\label{Translation}
\
\begin{enumerate}
\item[\rm (i)] $\tau_x$ acts continuously on the spaces $\mathscr{S}(\mathfrak{a}),C_c^\infty(\mathfrak{a})$ and $C^\infty(\mathfrak{a})$, equipped with the usual locally convex topologies.
\item[\rm (ii)] On $\mathscr{S}(\mathfrak{a})$, $\tau_x$ is given by
$$(\tau_xf)^{\wedge k}(y)=E(ix,y)\widehat{f}^{\,k}(y).$$
\item[\rm (iii)] $\tau_0=\id$ and $\tau_xf(y)=\tau_yf(x)$.
\item[\rm (iv)] If $f \in L^2(\mathfrak{a},\omega)$ has support in $B_r(0)\coloneqq \set{x \in \mathfrak{a}\mid |x|\le r}$, then
$$\mathrm{supp}(\tau_x f) \subseteq W.B_r(-x).$$
\item[\rm (v)] If $R$ is integral, i.e. $2\frac{\braket{\alpha,\beta}}{\braket{\alpha,\alpha}} \in \Z$ for all $\alpha,\beta \in R$, then $f\mapsto \tau_xf(y)$ is a distribution with support contained in $\mathrm{co}(W.x) + \mathrm{co}(W.y)$, where $\mathrm{co}$ denotes the convex hull.
\end{enumerate}
\end{lemma}

In particular, part (iv) will be a key tool in this article. It was proven in \cite{DH19} and is the indispensable part to prove the theorems on elliptic regularity.  Part (v) is proven in \cite{ASS10} for the case when $R$ spans $\mathfrak{a}$, but it is also true in the case where $R$ does not span $\mathfrak{a}$. To see this, put $\mathfrak{b}\coloneqq \mathrm{span}_\R R$ and $\mathfrak{c}\coloneqq \mathfrak{b}^\perp$. Let $V_\mathfrak{b}$ and $\mathfrak{\tau}^\mathfrak{b}_x$ be Dunkl's intertwining operator and generalized translations on $\mathfrak{b}$ associated to $(R,k)$, respectively. From the uniqueness of Dunkl's intertwining operator we obtain that
$$V=V_{\mathfrak{b}} \otimes \id_{\mathfrak{c}}.$$
If $S_y$ denotes the usual translation operator $S_yf(x) =f(x+y)$ for $f:\mathfrak{c} \to \C$, we see from the above factorization of $V$ that
\begin{equation}\label{TranslationFreeDirection}
\tau_{x} = \tau_{x_b}^{\mathfrak{b}}\otimes S_{x_c}, \quad x=x_b+x_c \te{ with } x_b \in \mathfrak{b},\, x_c \in \mathfrak{c}.
\end{equation}
Therefore, the result from \cite{ASS10} are also true if $R$ does not span $\mathfrak{a}$.

\section{Dunkl convolution}\label{Convolution}
The Dunkl convolution is already known and in the literature convolutions of functions and distributions are studied, see for instance \cite{OS05}. Later in this paper, we will need the convolution of two distributions with non-compact support. To our knowledge, this has not been studied so far, only the case where one of the distributions has compact support is known, see for instance \cite{OS05}. For open $\Omega \subseteq \mathfrak{a}$, we denote by $\mathcal{D}'(\Omega)$ and $\mathcal{E}'(\Omega)$ the spaces of distributions on $\mathfrak{a}$ with support contained in $\Omega$ and compact support contained in $\Omega$, respectively. Both spaces are equipped with the topology of pointwise convergence. As usual, the evaluation of a distribution $u \in \mathcal{D}'(\Omega)$ in $\varphi \in C_c^\infty(\Omega)$ will be denoted by the pairing
$$\langle u,\varphi\rangle\coloneqq u(\varphi).$$ 
To any locally integrable function $f:\Omega \to \C$ we associate a distribution $u_f^k \in \mathcal{D}'(\Omega)$ by the assignment
$$\langle u_f^k,\varphi \rangle \coloneqq \langle f,\overline{\varphi}\rangle_{\omega}= \int_{\Omega}f(x)\varphi(x)\omega(x)\dif x.$$
This embedding of locally integrable functions into distributions makes it compatible with the Dunkl setting and differs from the usual embedding, which is the reason to use a superscript $k$ in the notation. In fact, if $\Omega$ is $W$-invariant, the Dunkl operators act continuously on $\mathcal{D}'(\Omega)$ by
$$\langle T_\xi u,\varphi \rangle \coloneqq -\langle u,T_\xi\varphi \rangle,$$
so that the skew-symmetry of Dunkl operators in $L^2(\mathfrak{a},\omega)$, cf. \cite{R03}, leads to
$$T_\xi u_f^k= u_{T_\xi f}^k, \te{ for all } f \in C^\infty(\Omega).$$
Moreover, smooth functions $m \in C^\infty(\Omega)$ act continuously on $\mathcal{D}'(\Omega)$ by multiplication, namely
$$\langle m\cdot u,\varphi\rangle \coloneqq \langle u,m\varphi \rangle, \te{ so that } m\cdot u_f^k=u_{mf}^k.$$
In order to study the Dunkl convolution of distributions, we introduce the following sets in $\mathfrak{a}\times \mathfrak{a}$. \\
\begin{minipage}{0.6\linewidth}
For $r>0$ we define
$$D_r^W\coloneqq \bigcup\limits_{w \in W}\set{(x,y) \in \mathfrak{a}\times \mathfrak{a} \mid |x+wy|\le r}.$$
This set is invariant under the canonical action of $W\times W$ on $\mathfrak{a}\times \mathfrak{a}$. In fact,
it is a $W\times W$-orbit of a diagonal of width $r$ in $\mathfrak{a} \times \mathfrak{a}$. \\
In rank one, with $R=\set{\pm 1} \subseteq \R$, we have $W=\set{\pm \id}$ and $D_r^W$ in $\R^2$ is visualized on the right.
\end{minipage}
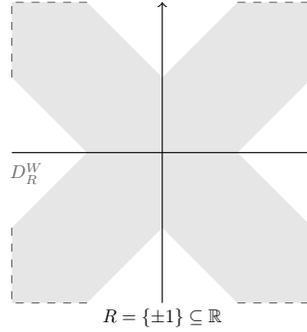
\begin{minipage}{0.39\linewidth}
\begin{center}
  \begin{tikzpicture}
 	\fill[color=gray!20!white] 
 	(0,-1) -- (1,-2) -- (2,-2) -- (2,-1) -- (1,0) --
 	(2,1) -- (2,2) -- (1,2) -- (0,1) -- (-1,2) --
 	(-2,2) -- (-2,1) -- (-1,0) -- (-2,-1) --
 	(-2,-2) -- (-1,-2) -- (0,-1);
  	\draw[->,black] (-2,0) -- (2,0);
  	\draw[->,black] (0,-2) -- (0,2);
  	\draw[-, dashed, color=gray!80!black] (-2,1)--(-2,2);
  	\draw[-, dashed, color=gray!80!black] (-2,-1)--(-2,-2);
  	\draw[-, dashed, color=gray!80!black] (2,1)--(2,2);
  	\draw[-, dashed, color=gray!80!black] (2,-1)--(2,-2);
  	\draw[-, dashed, color=gray!80!black] (1,-2)--(2,-2);
  	\draw[-, dashed, color=gray!80!black] (-1,-2)--(-2,-2);
  	\draw[-, dashed, color=gray!80!black] (1,2)--(2,2);
  	\draw[-, dashed, color=gray!80!black] (-1,2)--(-2,2);
    \node[anchor=south east, color=gray!80!black, scale=0.7] at (-1.5,-0.5) {$D_R^W$};
    \node[color=black, scale=0.7] at (0,-2.2) {$R=\set{\pm 1} \subseteq \R$};
  \end{tikzpicture}
\end{center}
\end{minipage}

\begin{definition}
We call two distributions $u,v \in \mathcal{D}'(\mathfrak{a})$ $W$-convolvable if for each $r>0$ the intersection $\mathrm{supp}(u\otimes v)\cap D_r^W$ is bounded, i.e. compact. \\
Here $u\otimes v$ is the usual tensor product of $u$ and $v$. Note that $\mathrm{supp}(u\otimes v)=\mathrm{supp}\, u \times \mathrm{supp} \, v$, so the distributions $u,v$ are $W$-convolvable in the following cases:
\begin{enumerate}
\item $u$ or $v$ has compact support.
\item the supports of $u$ and $v$ are contained in a $W$-invariant closed convex cone $C$ which is proper, i.e. $C$ does not contain one-dimensional subspaces. 
\end{enumerate}
\end{definition}

\begin{remark}
We note the following:
\begin{enumerate}
\item $u,v \in \mathcal{D}'(\mathfrak{a})$ are $W$-convolvable iff the restriction of $+:\mathfrak{a} \times \mathfrak{a} \to \mathfrak{a}$ to $(W.\mathrm{supp}\, u) \times (W.\mathrm{supp}\, v)$ is a proper map.
\item Even in rank one there exist distributions $u$ and $v$ with non-compact support which are $W$-convolvable. As an example, for $R=\set{\pm 1} \tm \R$, consider the distributions
$$u=\sum\limits_{n \in \N} \delta_{2^{2n}} \; \te{ and } \; v=\sum\limits_{n \in \N} \delta_{2^{2n+1}}$$
with supports $2^{2\N}$ and $2^{2\N+1}$, respectively. Then $D_r^W \cap (2^{2\N}\times 2^{2\N+1})$ is always finite, so that $u$ and $v$ are $W$-convolvable.
\item A non-zero $W$-invariant proper closed convex cone $C$, does not have to exist. In fact, such a cone exists if and only if $R$ does not span $\mathfrak{a}$. Later in this paper, the case $\mathfrak{a}=\R^n$, $R=A_{n-1}$ and $C=[0,\infty)^n$ will be of high relevance.
\end{enumerate}
\end{remark}

\begin{lemma}\label{TranslationAsOperatorOfTwoVariables}
The Dunkl translation associated to $(R,k)$ defines a continuous linear operator
$$\tau:C^\infty(\mathfrak{a}) \to C^\infty(\mathfrak{a}\times\mathfrak{a}), \quad \tau \varphi(x,y)=\tau_x\varphi(y)=\tau_y\varphi(x).$$
Moreover, for $\varphi \in C_c^\infty(\mathfrak{a})$ with $\mathrm{supp}\, \varphi \subseteq B_r(0)$ we have
$$\mathrm{supp}(\tau\varphi) \subseteq D_r^W.$$
\end{lemma}

\begin{proof}
Since $\tau$ can be expressed in terms of Dunkl's intertwining operator $V$ and the operator $S$ defined by $Sf(x,y)=f(x+y)$, i.e.
$$\tau= (V\otimes V) \circ S \circ V^{-1},$$
the continuity is a consequence of Lemma \ref{Translation}. For $(x,y) \in \mathfrak{a}\times\mathfrak{a}$ we have $(x,y) \in D_r^W$ iff $y \in W.B_r(-x)$, so that $\mathrm{supp}(\tau\varphi) \subseteq D_r^W$ holds by Lemma \ref{Translation}.
\end{proof}

\begin{definition}
Assume that $u,v \in \mathcal{D}'(\mathfrak{a})$ are $W$-convolvable. Choose a cut-off function $\rho \in C^\infty(\mathfrak{a}\times \mathfrak{a})$ with support in an $\epsilon$-neighborhood of $\mathrm{supp} \, u \times \mathrm{supp} \, v$ and $\rho \equiv 1$ in a smaller neighborhood. Note that in this case $(\mathrm{supp}\, \rho) \cap D_r^W$ is still compact for all $r>0$. Then, we are able to define
\begin{equation}\label{DefConv}
\braket{u*_kv,\varphi}\coloneqq \braket{u\otimes v,\rho \cdot \tau \varphi}, \quad \varphi \in C_c^\infty(\mathfrak{a}),
\end{equation}
which does not depend on the particular choice of $\rho$. It is called the Dunkl convolution of $u$ and $v$. 
\end{definition}

This definition was already given in \cite{OS05} under that assumption that $u$ or $v$ has compact support.

\begin{theorem}\label{PropertiesDunklConvolution}
Consider $u_1,u_2,u,v \in \mathcal{D}'(\mathfrak{a})$, $\lambda \in \C$ and $\xi \in \mathfrak{a}$. Then:
\begin{enumerate}
\item[\rm (i)] If $u,v$ are $W$-convolvable, then $u*_kv \in \mathcal{D}'(\mathfrak{a})$ and
$$u*_kv = v*_k u.$$
Moreover, $T_\xi u$ and $v$ are $W$-convolvable, and so are $u$ and $T_\xi v$, with
$$T_\xi(u*_kv)=(T_\xi u)*_k v = u*_k(T_\xi v).$$
\item[\rm (ii)] If both $u_1$ and $u_2$ are $W$-convolvable with $v$, then $u_1+\lambda u_2$ is $W$-convolvable with $v$ and
$$(u_1+\lambda u_2)*_k v = (u_1*_kv)+\lambda(u_2*_kv).$$
\item[\rm (iii)] $u$ is $W$-convolable with the Dirac distribution $\delta_0=(\varphi\mapsto \varphi(0))$ and 
$$u*_k\delta_0=\delta_0*_k u = u.$$
\end{enumerate}
\end{theorem}

\begin{proof}
Everything is straightforward to verify, we only have to justify the formula for the action of the Dunkl operators on a convolution. For this we choose a more explicit $\rho$ in \eqref{DefConv}, namely
$$\rho(x,y)\coloneqq \rho_u(x)\rho_v(y), \quad x,y \in \mathfrak{a}$$
with $W$-invariant $\rho_u$ and $\rho_v$. Moreover, we choose $\rho_u$ such that it has support in an $\epsilon$-neighborhood of $\mathrm{supp}\, u$ and $\rho_u \equiv 1$ in a smaller neighborhood. We choose $\rho_v$ in a similar fashion. By Lemma \ref{Translation}, the choice of $\rho$ and the Leibniz formula $T_\xi(\chi f)=(\partial_\xi\chi)\cdot f+\chi\cdot (T_\xi f)$ for $W$-invariant $\chi$ we have
\begin{align*}
\braket{T_\xi (u*_k v),\varphi} &= -\braket{u\otimes v, \rho \cdot \tau T_\xi \varphi}= -\braket{u\otimes v, \rho  \cdot T_\xi^x(\tau\varphi)} \\
&=\braket{u\otimes v, (\partial_\xi^x \rho)\cdot (\tau\varphi)}-\braket{u\otimes v, T_\xi^x(\rho \cdot \tau\varphi)}  \\
&= -\braket{u\otimes v, T_\xi^x(\rho \cdot \tau\varphi)} = \braket{(T_\xi u)\otimes v, \rho \cdot \tau\varphi}=\braket{(T_\xi u )*_k v,\varphi}.
\end{align*}
\end{proof}

\begin{definition}
Similar to the Euclidean case ($k=0$), we define
$$(f*_kg)(x)\coloneqq \int_{\mathfrak{a}} (\tau_y f)(-x) g(x)\omega(x)\dif x,$$
for $f,g \in C^\infty(\mathfrak{a})$, one with compact support, or both $f,g \in \mathscr{S}(\mathfrak{a})$. Moreover, we define
$$(f*_k u)(x)\coloneqq \braket{u(y),(\tau_xf)(-y)},$$
for $f \in C^\infty(\mathfrak{a}), u \in \mathcal{D}'(\mathfrak{a})$, one with compact support. Here $u(y)$ means that $u$ acts on functions of the $y$-variable.
\end{definition}

The following properties are straightforward, or can be found in \cite{OS05}.

\begin{lemma}\label{DifferentConvolutions}
The Dunkl convolution satisfies:
\begin{enumerate}
\item[\rm (i)] For $f,g \in \mathscr{S}(\mathfrak{a})$ one has $f*_k g \in \mathscr{S}(\mathfrak{a})$ and $(f*_k g)^{\wedge k}=\widehat{f}^{\,k}\cdot \widehat{g}^{\,k}$.
\item[\rm (ii)] For $f \in C_c^\infty(\mathfrak{a})$, the map $g \mapsto f*_k g$ is continuous on $C^\infty(\mathfrak{a})$ and satisfies
\begin{align*}
f*_kg&=g*_k f, \\
u_f^k *_k u_g^k &= u_{f*_k g}^k.
\end{align*}
\item[\rm (iii)] For $f \in C^\infty(\mathfrak{a})$ and $u \in \mathcal{D}'(\mathfrak{a})$, one with compact support, we have that $f*_k u \in C^\infty(\mathfrak{a})$ and
$$u_{f*_k u}^k = u_f^k *_k u.$$
\end{enumerate}
\end{lemma}

\begin{proposition}\label{Continuity}
Let $\Omega_1,\Omega_2 \subseteq \mathfrak{a}$ be open $W$-invariant sets such that $\overline{\Omega_1\times \Omega_2} \cap D_r^W$ is compact for all $r >0$. Then the Dunkl convolution defines a sequentially continuous operator
$$*_k:\mathcal{D}'(\Omega_1)\times \mathcal{D}'(\Omega_2) \to \mathcal{D}'(\mathfrak{a}).$$
\end{proposition}

\begin{proof}
This is an immediate consequence of the continuity of the tensor product and Definition \ref{DefConv}, since $\rho$ can be chosen uniformly for all $(u,v) \in \mathcal{D}'(\Omega_1)\times\mathcal{D}'(\Omega_2)$. 
\end{proof}

\begin{corollary}\label{Density}
The distributions $u_f^k$, $f \in C_c^\infty(\mathfrak{a})$, form a dense subspace of $\mathcal{D}'(\mathfrak{a})$.
\end{corollary}

\begin{proof}
Since $\mathcal{E}'(\mathfrak{a})$ is dense in $\mathcal{D}'(\mathfrak{a})$, it suffices to verify density in $\mathcal{E}'(\mathfrak{a})$. 
To do so, choose a non-negative $\psi \in C_c^\infty(\mathfrak{a})$ with $\mathrm{supp}\, \psi \subseteq B_1(0)$ and $\nrm{\psi}_{L^1(\mathfrak{a},\omega)}=1$. For $\epsilon>0$ we put
$$\psi_\epsilon(x)\coloneqq \frac{1}{\epsilon^\gamma}\psi\left(\frac{x}{\epsilon}\right), \quad \te{ with } \gamma\coloneqq \dim \, \mathfrak{a}+\frac{1}{2}\sum\limits_{\alpha \in R} k_\alpha.$$
Then, $\mathrm{supp}\, \psi_\epsilon \subseteq B_\epsilon(0)$ and $\nrm{\psi_\epsilon}_{L^1(\mathfrak{a},\omega)}=1$. Moreover, $u_{\psi_\epsilon}^k$ tends to $\delta_0$ pointwise since
$$|\braket{\delta_0-u_{\psi_\epsilon}^k,\varphi}| \le \int_{\mathfrak{a}}\psi_\epsilon(x)|\varphi(0)-\varphi(x)|\omega(x)\; \mathrm{d}x \le \nrm{\psi-\psi(0)}_{\infty,B_\epsilon(0)}.$$
For $u \in \mathcal{E}'(\mathfrak{a})$ we know from Lemma \ref{DifferentConvolutions} that 
$$u*_k \psi_\epsilon \in C^\infty(\mathfrak{a}).$$ 
Moreover, $\supp (\tau_x\psi_\epsilon) \tm W.B_\epsilon(-x)$ and $u \in \mathcal{E}'(\mathfrak{a})$ lead to $u*_k\psi_\epsilon \in C_c^\infty(\mathfrak{a})$ by definition of the convolution.
Finally, Proposition \ref{Continuity} leads to
$$u_{u*_k\psi_\epsilon}^k = u *_k u_{\psi_\epsilon}^k \underset{\epsilon \to 0}{\longrightarrow} u *_k \delta_0=u$$
\end{proof}

\section{(Singular-)support of Dunkl convolutions}\label{Supports}
As already mentioned, our embedding $f\mapsto u_f^k$ of locally integrable functions into $\mathcal{D}'(\mathfrak{a})$ differs from the usual embedding. Thus, we shall define a specific notion of singular support, adapted to this embedding, and hence adapted to the Dunkl setting.

\begin{definition}
For $u \in \mathcal{D}'(\mathfrak{a})$ we define the $k$-singular support of $u$ as the complement of the largest open subset of $\mathfrak{a}$ on which $u$ is of the form $u_f^k$ for some smooth $f$. To be more precise,
$$\mathrm{singsupp}_k u \coloneqq \bigcap\limits_{\substack{\Omega \subseteq \mathfrak{a} \\ \te{open}}} \set{\mathfrak{a}\backslash \Omega \mid u|_\Omega=u_f^k \te{ for some } f \in C^\infty(\Omega)}.$$
\end{definition}
It is obvious that
$$\mathrm{singsupp}_k u \subseteq \mathrm{supp}\, u.$$
Moreover, the usual singular support $\mathrm{singsupp}\, u=\mathrm{singsupp}_0u$ differes from the $k$-singular support only by singular elements, namely
$$\mathrm{singsupp}_k u \cup \mathfrak{a}_{\mathrm{sing}} = \mathrm{singsupp}\; u \cup \mathfrak{a}_{\mathrm{sing}},$$
where $\mathfrak{a}_{\mathrm{sing}}\coloneqq \bigcup_{\alpha \in R} \alpha^\perp$ is the singular set in $\mathfrak{a}$.

\begin{lemma}\label{Support}
Let $u,v \in \mathcal{D}'(\mathfrak{a})$ be $W$-convolvable distributions. Then:
\begin{enumerate}
\item[\rm (i)] $\mathrm{supp}\, u_f^k = \mathrm{supp}\, f$ for all locally integrable $f:\mathfrak{a}\to \C$.
\item[\rm (ii)] For distributions $u,v \in \mathcal{D}'(\mathfrak{a})$ with $\mathrm{supp} \; u \subseteq B_r(0)$, $r>0$, we have
$$\mathrm{supp}(u*_kv) \subseteq B_r(0) + W.\mathrm{supp} \; v.$$
\item[\rm (iii)] If $R$ is an integral root system, then
$$\mathrm{supp}(u*_k v) \subseteq \overline{\mathrm{co}(W.\mathrm{supp}\, u)+\mathrm{co}(W.\mathrm{supp}\, v)}=\overline{\mathrm{co}(W.\mathrm{supp}\, u+W.\mathrm{supp}\,v)}.$$
\end{enumerate}
\end{lemma}

\begin{proof}
Part (i) is immediate by definition and the fact that the zero set of $\omega$ is a finite union of the hyperplanes. 
\begin{enumerate}
\item[(ii)] By part (i), Lemma \ref{DifferentConvolutions}, Proposition \ref{Continuity} and Corollary \ref{Density}, it suffices to prove
\begin{equation}\label{ConvFuncSupp}
\mathrm{supp}(f*_kg) \subseteq B_r(0) + W.\mathrm{supp}\, g
\end{equation}
for $f,g \in C_c^\infty(\mathfrak{a})$ with $\mathrm{supp}\, f \subseteq B_r(0)$. Since $\mathrm{supp} \, f \subseteq B_r(0)$, we have by Lemma \ref{Translation} (iv) that $\mathrm{supp} \; \tau_xf \subseteq W.B_r(-x)$. So that \eqref{ConvFuncSupp} is a consequence of the explicit formula
$$(f*_k g)(y)= \int_{\mathfrak{a}} (\tau_yf)(-x) g(x) \omega(x) \dif x.$$
\item[(iii)] For abbreviation, we write $A\coloneqq\mathrm{supp}\, u$ and $B\coloneqq \mathrm{supp}\, v$. Consider $\varphi \in C_c^\infty(\mathfrak{a})$ with $\mathrm{supp}\, \varphi \cap \overline{(\mathrm{co}(W.A)+\mathrm{co}(W.B))}=\emptyset$. By Lemma \ref{Translation} (v) we have $\tau\varphi(x,y)=0$ for all $(x,y) \in A\times B$ and therefore $\braket{u*_kv,\varphi}=0$.
\end{enumerate}
\end{proof}

\begin{corollary}\label{DistConeAlg}
Assume that $R$ is integral and $C\subseteq \mathfrak{a}$ is a proper $W$-invariant closed convex cone. Then the space of distributions with support contained in $C$ is a unital , associative and commutative algebra over $\C$ with the Dunkl convolution as multiplication.
\end{corollary}

\begin{proof}
By Lemma \ref{Support} (iii) and the conditions on $C$, we see that $\mathrm{supp}(u*_kv) \subseteq C$ for all $u,v \in \mathcal{D}'(\mathfrak{a})$ with support contained in $C$. Theorem \ref{PropertiesDunklConvolution} shows that we have a commutative algebra over $\C$. Moreover, on the Schwartz space $\mathscr{S}(\mathfrak{a})$, the Dunkl convolution is associative. Hence, Lemma \ref{DifferentConvolutions} and Corollary \ref{Density} show that $*_k$ is associative in general.
\end{proof}

\begin{theorem}\label{Singularsupport}
Let $u,v \in \mathcal{D}'(\mathfrak{a})$ be $W$-convolvable distributions. Then:
\begin{enumerate}
\item[\rm (i)] If $\mathrm{singsupp}_k u \subseteq B_r(0)$, then
$$\mathrm{singsupp}_k(u*_kv) \subseteq B_r(0)+W.\mathrm{singsupp}_k v.$$
\item[\rm (ii)] If $R$ is integral, then
$$\mathrm{singsupp}_k(u*_kv) \subseteq \overline{\mathrm{co}(W.\mathrm{singsupp}_k\, u)+\mathrm{co}(W.\mathrm{singsupp}_k v)}.$$
\end{enumerate}
\end{theorem}

\begin{proof}
We consider two cases.
\begin{enumerate}
\item[(a)] First, we assume that both $u$ and $v$ have compact support. Choose an arbitrary $\epsilon>0$ and cutoff functions $\chi_u,\chi_v \in C_c^\infty(\mathfrak{a})$ with
\begin{align*}
\mathrm{supp}\, \chi_u &\subseteq \mathrm{singsupp}_ku + B_\epsilon(0), \quad \chi_u \equiv 1 \te{ on } \mathrm{singsupp}_k u, \\
\mathrm{supp}\, \chi_v &\subseteq \mathrm{singsupp}_kv + B_\epsilon(0), \quad \chi_v \equiv 1 \te{ on } \mathrm{singsupp}_k v.
\end{align*}
Hence, we conclude that
\begin{align*}
\mathrm{supp}(\chi_uu) \subseteq \mathrm{singsupp}_ku + B_\epsilon(0), \quad (1-\chi_u)u=u_{f}^k,\; f \in C_c^\infty(\mathfrak{a}), \\
\mathrm{supp}(\chi_vv) \subseteq \mathrm{singsupp}_kv + B_\epsilon(0), \quad (1-\chi_v)v=u_{g}^k, \; g \in C_c^\infty(\mathfrak{a}). 
\end{align*}
Due to Lemma \ref{DifferentConvolutions} we have
\begin{align*}
((1-\chi_v)v)*_k(\chi_uu) &= u_{g}^k *_k (\chi_u u) = u_{g*_k (\chi_uu)}^k, \\
(\chi_vv)*_k((1-\chi_u)u) &= (\chi_v v)*_k u_{f}^k = u_{f*_k (\chi_vv)}^k, \\
((1-\chi_v)v)*_k((1-\chi_u)u) &= u_{f}^k * u_{g}^k = u_{f*_k g}^k,
\end{align*}
and therefore we see that
\begin{align*}
\mathrm{singsupp}_k (u*_kv) &= \mathrm{singsupp}_k((\chi_uu)*_k(\chi_vv)) \\
&\subseteq \mathrm{supp} ((\chi_u u)*_k (\chi_vv)).
\end{align*}
Finally, we distinguish between the two situations in the theorem:
\begin{enumerate}
\item[(i)] In this case, we can conclude that
$$\mathrm{supp}(\chi_uu) \subseteq \mathrm{singsupp}_ku + B_\epsilon(0) \subseteq B_{r+\epsilon}(0)$$
and therefore 
$$\mathrm{singsupp}_k(u*_kv) \subseteq B_{r+\epsilon}(0)+W.(\mathrm{singsupp}(v)+B_\epsilon(0))$$
by Theorem \ref{Support} (ii). Since $\epsilon>0$ was arbitrary, the claim holds.
\item[(ii)] In this case, we conclude from Theorem \ref{Support} (iii) that
$$\mathrm{singsupp}_k(u*_kv) \subseteq \overline{\mathrm{co}(W.A_{\epsilon})+\mathrm{co}(W.B_{\epsilon})}$$
with
\begin{align*}
A_{\epsilon}:= \mathrm{singsupp}_k u+B_\epsilon(0), \\
B_{\epsilon}:= \mathrm{suppsupp}_k v+B_\epsilon(0).
\end{align*}
Since $\epsilon>0$ was arbitrary, the claim follows.
\end{enumerate}
\item[(b)] For arbitrary $u,v \in \mathcal{D}(\mathfrak{a})$ and $R>0$, choose a cutoff function $\chi \in C_c^\infty(\mathfrak{a})$ with $\chi \equiv 1$ on $B_{2R}(0).$
Then on easily obtains from Theorem \ref{Support} (ii) that
$$u*_kv|_{B_R(0)} =(\chi u)*_k(\chi v)|_{B_{R}(0)}.$$
Application of step (a) above to the distributions $\chi u, \chi v$ finishes the proof, because $R>0$ was chosen arbitrarily.
\end{enumerate}
\end{proof}

\section{Hypoellipticity of elliptic Dunkl operators}\label{Hypoellipticity}
Let $\Omega \subseteq \mathfrak{a}$ be open. In contrast to the usual meaning, we say for $u \in \mathcal{D}'(\Omega)$ and a set $M$ of locally integrable functions that
$$u \in M \te{ iff } u=u_f^k \te{ with } f \in M.$$
This means that for all $\varphi \in C_c^\infty(\Omega)$ we have
$$\braket{u,\varphi}=\int_{\Omega} \varphi(x)f(x)\omega(x)\; \mathrm{d}x \quad \te{ with } f \in M.$$
It is important to keep this in mind. In particular, $u \in C^\infty(\mathfrak{a})$ means that $u$ is given (in the usual sense) by $f\omega$ with some $f \in C^\infty(\mathfrak{a})$. Hence, $u \in C^\infty(\mathfrak{a})$ does not mean that $u$ is a smooth function in the usual sense, since $\omega$ is in general not smooth along $\mathfrak{a}_{\mathrm{sing}}$. But, as already mentioned, our notion is adapted to the Dunkl setting with the advantage that things are getting similar to the usual theory of elliptic differential operators. \\
Recall that the Dunkl transform of a tempered distribution $u \in \mathscr{S}'(\mathfrak{a})$ is defined as
$$\braket{\widehat{u}^k,f}\coloneqq \braket{\mathcal{F}_ku,f}\coloneqq \braket{u,\widehat{f}^{\,k}}, \quad f \in \mathscr{S}(\mathfrak{a}),$$
so that for $f \in \mathscr{S}(\mathfrak{a})$ we have 
$$\mathcal{F}_ku_f^k = u_{\mathcal{F}_kf}^k.$$

In this section we are interested in the study of Dunkl operators, which are elliptic in the following sense.

\begin{definition}
A Dunkl operator $p(T)$ is called elliptic of degree $m\in \N_0$ if $p=\sum_{n = 0}^m p_n$ with $p_n \in \mathcal{P}_n$ satisfies $p_m(x) \neq 0$ for all $x \in \mathfrak{a}\backslash\set{0}$.
\end{definition}

For $k=0$, an elliptic Dunkl operator is nothing but an elliptic differential operator with constant coefficients. For instance, the Dunkl Laplacian
$$\Delta_k\coloneqq \braket{T,T}=\Delta_{\mathfrak{a}}+\sum\limits_{\alpha \in R} k_\alpha \left( \frac{\partial_\alpha}{\braket{\alpha,\m}} - \frac{|\alpha|^2}{2}\frac{1-s_\alpha}{\braket{\alpha,\m}^2}\right),$$
where $\Delta_{\mathfrak{a}}$ is the Laplacian on $\mathfrak{a}$, is an elliptic Dunkl operator of degree $2$. \\
As usual, we put
$$\braket{x}\coloneqq \sqrt{1+|x|^2} $$
as function on $\mathfrak{a}$.

\begin{proposition}\label{DifferentationByFourierTrafo}
Consider $f \in L^1(\mathfrak{a},\omega)$ such that $x\mapsto\braket{x}^\ell f(x) \in L^1(\mathfrak{a},\omega)$ for some fixed $\ell \in \N_0$. Then,  $\widehat{f}^{\,k} \in C^\ell(\mathfrak{a})$.
\end{proposition}

\begin{proof}
For $\ell=0$, this is just the Riemann-Lebesgue Lemma for the Dunkl transform, cf. Lemma \ref{DunklTrafo}. Otherwise, it is a consequence of usual theorems on differentiable parameter integrals and the estimate 
$$|\partial_\xi^\beta E_k(-ix,\xi)| \le |x|^{|\beta|} \le \braket{x}^\ell, \quad \te{ for all }\beta \in \N_0^n,\, |\beta|\le \ell,$$
where the first inequality can be found in \cite[Proposition 2.6]{R03}.
\end{proof}

Let $a \in C^\infty(\mathfrak{a})$ be a smooth function and $m \in \R$. Assume that for all $\beta \in \N_0^n$ there exists a constant $C_\beta \ge 0$ with
\begin{equation}\label{ClassicalSymbol}
\abs{\partial^\beta a (x)}\le C_\beta \braket{x}^{m-\abs{\beta}},
\end{equation}
then it is well known that the (distributional) Fourier transform of $a$ has singular support contained in $\set{0}$, cf. \cite[Proof of Theorem 7.1.22]{H03}. The following lemma is a generalization of this to the case of arbitrary $k\ge 0$.

\begin{lemma}\label{DunklTrafoOfSymbolds}
Assume that for $a \in C^\infty(\mathfrak{a})$ there exists some $m \in \R$ such that
\begin{equation}\label{SymbolClass}
|T^\beta a(x)| \le C_\beta\braket{x}^{m-|\beta|}
\end{equation}
for some constants $C_\beta$ and all $x \in \mathfrak{a},\beta \in \N_0^n$. Then
$$\mathrm{singsupp}_k(\mathcal{F}_ku_a^k) \subseteq \set{0}.$$
\end{lemma}

\begin{proof}
Note that $u_a^k$ is tempered, as $a$ is at most of polynomial growth.
For each $\ell \in \N_0$, we can find $N \in \N$ such that
$$x\mapsto \braket{x}^\ell \m T^\beta a(x) \in L^1(\mathfrak{a},\omega_k) \te{ for all } \beta \in \N_0^n, \, |\beta|\ge N.$$
By Proposition \ref{DifferentationByFourierTrafo}, we obtain that
$$\mathcal{F}_k(T^\beta a) \in C^\ell(\mathfrak{a}),$$
so that as distributions
$$(ix)^\beta \mathcal{F}_ku_a^k = \mathcal{F}_k(T^\beta u_a^k) \in C^\ell(\mathfrak{a}).$$
Therefore, we can conclude that
$$(\mathcal{F}_ku_a^k)|_{\mathfrak{a}\backslash \set{0}} \in C^\ell(\mathfrak{a}\backslash\set{0}).$$
But $\ell \in \N_0$ was arbitrary, so $\mathrm{singsupp}_k(\mathcal{F}_ku_a^k)\subseteq \set{0}$.
\end{proof}

Assume that $a \in C^\infty(\mathfrak{a})$ satisfies \eqref{ClassicalSymbol} and $\abs{a(x)} \ge C\abs{x}^m$ for large $x$ and some constant $C$. Then the reciprocal $\frac{1}{a}$ satisfies \eqref{ClassicalSymbol} with $-m$ instead of $m$. This is an immediate consequence of the quotient rule for partial derivatives. Dunkl operators does not have a general Leibniz rule, so there is no quotient rule. To avoid this, we can use the subsequent lemma.

\begin{proposition}\label{DunklDerivativeRationalFunctions}
Let $f=\frac{q}{p} \in C^\infty(\Omega)$ be a rational function on some $W$-invariant open $\Omega \subseteq \mathfrak{a}$, i.e. $p,q \in \mathcal{P}$. Then, $T_\xi(k)f$ is rational on $\Omega$ for all $\xi \in \mathfrak{a}$. To be more precise, there exist finitely many polynomials $\tilde{q_i}\in \mathcal{P}$ of degree $\deg\, p + \deg\, q-1$ and $w_i \in W$, such that for all $x \in \Omega$
$$T_\xi f(x)=\sum_{i} \frac{\tilde{q_i}(x)}{p(x)p(w_ix)}.$$
\end{proposition}

\begin{proof}
Rewriting the difference part of the Dunkl operator, we observe the following
\begin{align*}
T_\xi f(x) &=\frac{\partial_\xi q(x)\m p(x) - q(x)\m \partial_\xi p(x)}{p(x)^2} \\
&\quad + \sum\limits_{\alpha \in R_+}  \frac{k_\alpha \braket{\alpha,\xi}}{p(x)p(s_\alpha x)} \frac{q(x)p(s_\alpha x)-q(s_\alpha x)p(x)}{\braket{\alpha,x}}.
\end{align*}
But the polynomial $q(x)p(s_\alpha x)-q(s_\alpha x)p(x)$ vanishes on $\alpha^\perp$, hence it is divisible by $\braket{\alpha,x}$ and the claim holds.
\end{proof}

\begin{lemma}\label{PolynomialsAreSymbols}
Let $p \in \mathcal{P}$ be a polynomial of degree $m \in \N_0$. Then:
\begin{enumerate}
\item[\rm (i)] $|T^\beta p(x)|\le C_\beta \braket{x}^{m-|\beta|}$ for all $\beta \in \N_0^n$ and some constant $C_\beta \ge 0$.
\item[\rm (ii)] If $p=\sum_{j=0}^m p_j$ with $p_j$ homogeneous of degree $j$ and $p_m(x)\neq 0$ for all $x \in \mathfrak{a}\backslash\set{0}$, then there exists some $R>0$ and $q \in C^\infty(\mathfrak{a})$ with $pq\equiv 1$ on $\mathfrak{a}\backslash B_R(0)$. Moreover, 
\begin{equation}\label{SymbolReziprocPoly}
|T^\beta q(x)|\le C_\beta \braket{x}^{-m-|\beta|}
\end{equation}
for all $\beta \in \N_0^n$ and some constant $C_\beta$.
\end{enumerate}
\end{lemma}

\begin{proof}
\
\begin{enumerate}
\item This is obvious, since $T_\xi$ is homogeneous of degree $-1$ and 
$$|x^\alpha|\le |x|^{|\alpha|} \le \braket{x}^{|\beta|}$$
for all $\alpha,\beta \in \N_0$ with $|\alpha|\le|\beta|$.
\item Since $p_m(x)$ is homogeneous of degree $m$, there exists some $c>0$ with $p_m(x)\ge c|x|^m$. Therefore,  
\begin{equation}\label{LowerBoundEllipticPolynomials}
|p(x)|>\tfrac{c}{2}|x|^m
\end{equation} 
for all $x \in \mathfrak{a}\backslash B_R(0)$ and some large $R>0$. Then, choose $q \in C^\infty(\mathfrak{a})$ with $q=\tfrac{1}{p}$ on $\mathfrak{a}\backslash B_R(0)$. By iteration of Proposition \ref{DunklDerivativeRationalFunctions}, we can find for $\beta \in \N_0^n$ finitely many polynomials $\tilde{q_i} \in \mathcal{P}$ of degree 
$$\deg \tilde{q_i} = (2^{|\beta|}-1)\deg p - |\beta|=(2^{|\beta|}-1)m-|\beta|$$
and $w_{i,j}\in W, j=1,\ldots ,2^{|\beta|}$ satisfying
$$(T^\beta q)(x) = \sum\limits_i \frac{\tilde{q_i}(x)}{p(w_{i,1}x)p(w_{i,2}x)\cdots p(w_{i,2^{|\beta|}}x)},$$
for all $x \in \mathfrak{a}\backslash B_R(0)$. By part (i) and estimate \eqref{LowerBoundEllipticPolynomials}, there exists $C'_\beta \ge 0$ with
$$|(T^\beta q)(x)| \le C'_\beta \frac{\braket{x}^{(2^{|\beta|}-1)m-|\beta|}}{\braket{x}^{2^{|\beta|}\cdot m}} = C_\beta' \braket{x}^{-m-|\beta|},$$
for all $x \in \mathfrak{a}\backslash B_R(0)$.
Finally, as $q$ is continuous, estimate \eqref{SymbolReziprocPoly} follows.
\end{enumerate}
\end{proof}

The subsequent proof of the theorem on hypoellipticity follows basically the classical ideas as in \cite[Theorem 7.1.22]{H03}. Recall that for any distribution $u\in \mathcal{D}'(\mathfrak{a})$ we have $\delta_0*_k u=u$.

\begin{theorem}[Hypoellipticity]\label{EllipticRegularityI}
Let $p(T)$ be an elliptic Dunkl operator and $\Omega \subseteq \mathfrak{a}$ a $W$-invariant open subset. Then for $u \in \mathcal{D}'(\Omega)$, we have
$$W.\mathrm{singsupp}_k u = W.\mathrm{singsupp}_k(p(T)u).$$
\end{theorem}

\begin{proof}
Let $m$ be the degree of $p$. 
By Lemma \ref{PolynomialsAreSymbols}, we choose $q \in C^\infty(\mathfrak{a})$ with $p(-i\m)q\equiv 1$ on $\mathfrak{a}\backslash B_R(0)$ for some large $R>0$ and such that for all $x \in \mathfrak{a}$ and $\beta \in \N_0^n$
$$|T^\beta q(x)| \le C_\beta\braket{x}^{-m-|\beta|}$$
with some constant $C_\beta\ge 0$. Thus, Lemma \ref{DunklTrafoOfSymbolds} leads to a tempered distribution $E\coloneqq \frac{1}{c_k}\mathcal{F}_k^{-1}u_q^k \in \mathscr{S}'(\mathfrak{a})$ with $k$-singular support contained in $\set{0}$, where $c_k$ is the Macdonald-Metha constant \eqref{MacdonaldMetha}. We put
$$R\coloneqq \delta_0-p(T)E \in \mathscr{S}'(\mathfrak{a}).$$
The Dunkl transform of $R$ is
$$\mathcal{F}_kR = \frac{1}{c_k}u_1^k-\mathcal{F}_k(p(T)E) = \frac{1}{c_k}(u_1^k-u_{p(-i\m)q}^k) = u_{f}^k$$
with $c_k\m f=1-p(-i\m)q \in C_c^\infty(\mathfrak{a})$. Therefore, $R=u_{\mathcal{F}_k^{-1}f}^k$ and $\mathrm{singsupp}_kR=\emptyset$. \\
For $x_0 \in \Omega\backslash W.\mathrm{singsupp}_k(p(T)u)$ choose a cutoff function $\chi \in C_c^\infty(\mathfrak{a})$ such that $\chi \equiv 1$ in a neighborhood of $W.x_0$. We consider $\chi u$ as a compactly supported distribution on $\mathfrak{a}$ by extending it by $0$ outside of $\Omega$. Then
\begin{equation}\label{KeyConvolution}
\chi u = \delta_0 *_k (\chi u) = (p(T)E+R)*_k(\chi u) = E*_k(p(T)(\chi u)) + R*_k(\chi u).
\end{equation}
Since $\mathrm{singsupp}_k R=\emptyset$ and $\mathrm{singsupp}_k E \subseteq \set{0}$, we can use Theorem \ref{Singularsupport} to obtain
$$\mathrm{singsupp}_k(\chi u) = \mathrm{singsupp}_k(E*_k(p(T)(\chi u))) \subseteq W.\mathrm{singsupp}_k(p(T)(\chi u)).$$
But $\chi \equiv 1$ in a neighborhood of $W.x_0$, so that $p(T)\chi u = p(T)u$ and $\chi u = u$ near $W.x_0$ and therefore $x_0 \notin \mathrm{singsupp}_k u$. From this we have
$$\mathrm{singsupp}_k u \subseteq W.\mathrm{singsupp}_k (p(T)u).$$
Finally, it is obvious that $\mathrm{singsupp}_k(p(T)u) \subseteq W.\mathrm{singsupp}_ku$, which finishes the proof.
\end{proof}

\section{Elliptic regularity of Dunkl operators}
First, we give a review of Dunkl-type Sobolev spaces and their properties as studied in \cite{MT04}.
\begin{definition}
For $s \in \R$ the Dunkl-type Sobolev space of order $s$ is defined by
$$H_k^s(\mathfrak{a}) \coloneqq \set{u \in \mathscr{S}'(\mathfrak{a})\mid \braket{x}^s\widehat{u}^k \in L^2(\mathfrak{a},\omega)}.$$
Hence, $u \in H_k^s(\mathfrak{a})$ iff $\widehat{u}^k \in \braket{x}^{-s}L^2(\mathfrak{a},\omega)$ and we identity $\widehat{u}^k$ with the function $f \in \braket{x}^{-s}L^2(\mathfrak{a},\omega)$ such that $\widehat{u}^k=u_f^k$. Under this identification, the inner product on $H_k^s(\mathfrak{a})$ is defined as
$$\braket{u,v}_{H_k^s}\coloneqq \int_{\mathfrak{a}}\braket{x}^{2s}\widehat{u}^k(x)\overline{\widehat{v}^k(x)}\omega(x)\;\mathrm{d}x.$$
\end{definition}

For our purpose, we need the following results from \cite{MT04}.

\begin{theorem}\label{Sobolev}
The Sobolev spaces $H_k^s(\mathfrak{a})$ are Hilbert spaces, satisfying
\begin{enumerate}
\item[\rm (i)] Let $s \in \N_0$. Then up to identification of $u_f^k$ with $f$,
$$H_k^s(\mathfrak{a})=\set{f \in L^2(\mathfrak{a},\omega) \mid T^\alpha f \in L^2(\mathfrak{a},\omega) \te{ for all } \alpha \in \N_0^n,|\alpha|\le s}.$$
Moreover, an equivalent norm on $H_k^s(\mathfrak{a})$ is induced by the inner product
$$(f,g)\mapsto \sum\limits_{|\alpha|\le s} \int_{\mathfrak{a}} (T^\alpha f)(x)\overline{(T^\alpha g)(x)}\omega(x)\; \d x.$$
\item[\rm (ii)] Dunkl operators are continuous linear operators $T_\xi:H_k^s(\mathfrak{a}) \to H_k^{s-1}(\mathfrak{a})$.
\item[\rm (iii)] For $\psi \in C_c^\infty(\mathfrak{a})$, $u \mapsto \psi u$ is a continuous map from $H_k^s(\mathfrak{a})$ into itself.
\item[\rm (iv)] $\mathcal{E}'(\mathfrak{a}) \subseteq \bigcup_{s \in \R} H_k^s(\mathfrak{a})$.
\item[\rm (v)] For $p \in \N$ and $s \in \R$ with $s>p+\tfrac{n}{2}+\tfrac{1}{2}\sum_{\alpha \in R} k_\alpha$, the identification $u_f^k\mapsto f$ yields a continuous embedding
$$H_k^s(\mathfrak{a}) \hookrightarrow C^p(\mathfrak{a}),$$
by identifying $u_f^k \mapsto f$, is a continuous embedding.
\end{enumerate}
\end{theorem}

In order to formulate the theorem on elliptic regularity, we actually need local Sobolev spaces in the following sense.
\begin{definition}
For $W$-invariant open $\Omega \subseteq \mathfrak{a}$ and $s \in \R$ we define
$$H_{k,loc}^s(\Omega) \coloneqq \set{u \in \mathcal{D}'(\Omega) \mid \psi u \in H_k^s(\mathfrak{a}) \te{ for all } \psi \in C_c^\infty(\mathfrak{a})}.$$
\end{definition}

Note that $H_k^s(\mathfrak{a}) \subseteq H_k^t(\mathfrak{a})$ for $t\le s$. Thus $H_{k,loc}^s(\Omega) \subseteq H_{k,loc}^t(\Omega)$ and in particular 
$$H_{k}^0(\mathfrak{a})=L^2(\mathfrak{a},\omega)\; \textrm{ and } \; H_{k,loc}^0(\Omega)=L^2_{loc}(\Omega,\omega).$$
From \cite[Equation (2.13)]{OS05} we have the following result.

\begin{proposition}
Let $v \in \mathscr{S}'(\mathfrak{a})$ and $u \in \mathcal{E}'(\mathfrak{a})$ be a tempered and a compactly supported distributions, respectively. Then $\widehat{u}^k \in \mathscr{S}(\mathfrak{a})$, $u*_k v \in \mathscr{S}'(\mathfrak{a})$ and
$$(u*_k v)^{\wedge k}=\widehat{u}^k \cdot \widehat{v}^k,$$
where $\widehat{u}^k$ again is identified with $f \in \mathscr{S}(\mathfrak{a})$ such that $\widehat{u}^k=u_f^k$.
\end{proposition}

As an immediate corollary to this proposition, we obtain the following.

\begin{corollary}\label{SobolevConvolution}
Suppose that $a \in C^\infty(\mathfrak{a})$ satisfies estimate \eqref{SymbolClass} for some $m \in \R$. Then for all $v \in H_k^s(\mathfrak{a})\cap \mathcal{E}'(\mathfrak{a})$ we have
$$v*_ku_a^k \in H_k^{s+m}(\mathfrak{a}).$$
\end{corollary}

\begin{theorem}[Elliptic regularity]\label{EllipticRegularityII}
Let $p(T)$ be an elliptic Dunkl operator of degree $m$. Then for each $W$-invariant open $\Omega \tm \mathfrak{a}$ and $u \in \mathcal{D}'(\Omega)$
$$p(T)u \in H_{k,\te{loc}}^s(\Omega) \te{ if and only if } u \in H_{k,\te{loc}}^{s+m}(\Omega).$$
\end{theorem}

\begin{proof}
Obviously, we only have to prove that $p(T)u \in H_{k,\te{loc}}^{s}(\Omega)$ implies that $u \in H_{k,\te{loc}}^{s+m}(\Omega)$. Let $\psi \in C_c^\infty(\Omega)$ with $\mathrm{supp}\, \psi \subseteq B_r(0)$ and choose a $W$-invariant $\chi \in C_c^\infty(\mathfrak{a})$ such that $\chi \equiv 1$ on a neighborhood of $B_r(0)$.
Recall the distributions $E,R$ from the proof of Theorem \ref{EllipticRegularityI}. 
Consider $\psi u$ and $\chi u$ as elements of $\mathcal{D}'(\mathfrak{a})$. We obtain similarly to Theorem \ref{EllipticRegularityI}, more precisely from \eqref{KeyConvolution}
\begin{equation}\label{ERCutoff}
\psi u = \psi (\chi u) = \psi \cdot (E*_k (p(T)(\chi u))) + \psi \cdot (R*_k (\chi u)).
\end{equation}
We discuss the terms on the right-hand side separately.
\begin{enumerate}
\item As in the proof of Theorem \ref{EllipticRegularityI}, we have $R*_k(\chi u) \in C^\infty(\mathfrak{a})$. Thus, $\psi(R*_k(\chi u)) \in C_c^\infty(\mathfrak{a})$ and in particular 
$$\psi \cdot (R*_k (\chi u)) \in H_k^t(\mathfrak{a})$$
for all $t \in \R.$
\item We claim that
\begin{equation}\label{ERCutoffDunklOp}
p(T)(\chi u)=\chi \cdot p(T)u + v,
\end{equation}
for some $v \in \mathcal{E}'(\mathfrak{a})$ with $v \equiv 0$ in a neighborhood of $B_r(0)$. To see this, consider any $\xi \in \mathfrak{a}$. The $W$-invariance of $\chi$ leads to
$$T_\xi (\chi u) = \chi \cdot (T_\xi u) + (\partial_\xi \chi) \cdot u.$$
Since $\chi \equiv 1$ in a neighborhood of $B_r(0)$, the last term (and all its images under Dunkl operators) vanishes on this neighborhood. Thus, if we iterate this argument, we obtain the stated equation \eqref{ERCutoffDunklOp}.
\item Decomposing $p(T)(\chi u)$ according to \eqref{ERCutoffDunklOp}, we see that
$$E*_k (p(T)(\chi u)) = E*_k(\chi \cdot p(T) u) + E*_k v.$$
By Theorem \ref{Singularsupport} and the fact $\mathrm{singsupp}_kE \subseteq \set{0}$ we get
$$\mathrm{singsupp}_k (E*_k v) \subseteq  W.\mathrm{singsupp}_k v \subseteq  W.\mathrm{supp}\, v \subseteq \mathfrak{a}\backslash B_r(0).$$
Therefore, by $\mathrm{supp} \; \psi \subseteq B_r(0)$, we conclude $\psi (E*_k v) \in C_c^\infty(\mathfrak{a})$ and thus
$$\psi(E*_kv) \in H_k^t(\mathfrak{a})$$
for all $t \in \R$. By assumption, $\chi p(T) u \in H_k^s(\mathfrak{a}) \cap \mathcal{E}'(\mathfrak{a})$, so that with Lemma \ref{SobolevConvolution} we have 
$$\psi(E*(\chi p(T) u))+\psi(E*(\chi p(T) u)) + \psi(E*_kv) \in H_k^{s+m}(\mathfrak{a}).$$
\end{enumerate}
Putting things from (i) and (iii) together, we obtain from \eqref{ERCutoff} that
$$\psi u \in H_k^{s+m}(\mathfrak{a})$$
for all $\psi \in C_c^\infty(\Omega)$, i.e. $u \in H_{k,loc}^{s+m}(\Omega)$.
\end{proof}

From Theorem \ref{EllipticRegularityII}, we obtain the following corollary.

\begin{corollary}
Consider some $W$-invariant open $\Omega \subseteq \mathfrak{a}$, $u \in \mathcal{D}'(\Omega)$ and an elliptic Dunkl operator $p(T)$. Then $u \in C^\infty(\Omega)$ in the following cases
\begin{enumerate}
\item[\rm (i)] $p(T)^m u \in L_{loc}^2(\Omega,\omega)$ for all $m \in \N_0$.
\item[\rm (ii)] $p(T)^m u \in C(\Omega)$ for all $m \in \N_0$.
\item[\rm (iii)] $u$ is an eigendistribution of $p(T)$.
\end{enumerate}
\end{corollary}

\begin{proof}
\
\begin{enumerate}
\item[\rm (i)] We recall that $H_{k,loc}^0(\Omega)=L_{loc}^2(\Omega,\omega)$. Furthermore, we note that $p(T)^m$ is elliptic of degree $m \cdot \deg p$. Therefore, by Theorem \ref{EllipticRegularityI}, we have that $u \in H_{k,loc}^{m\cdot \deg p}(\Omega)$ for all $m \in \N$, i.e.
$$u \in \bigcap\limits_{s \in \R}H_{k,loc}^s(\Omega)=C^\infty(\Omega).$$
\item[\rm (ii)] This is obvious, since $C(\Omega) \tm L^2_{loc}(\Omega,\omega)$.
\item[\rm (iii)] Assume that $p(T)u=\lambda u$ for $\lambda \in \C$. But $\tilde{p}(T)\coloneqq p(T)-\lambda$ is an elliptic Dunkl operator satisfying $\tilde{p}(T)u=0$. Thus, by Theorem \ref{EllipticRegularityI} we conclude that $u \in C^\infty(\Omega)$.
\end{enumerate}
\end{proof}

\section{Convolution of type A Riesz distributions}
We conclude this paper with an application of the results from the previous sections, in particular the results from Section 3.
In this section we consider the root system
$$A_{n-1}\coloneqq \set{\pm (e_i-e_j) \mid 1 \le i < j \le n} \tm \R^n,$$
with the canonical basis $e_1,\ldots,e_n$ of $\R^n$. The associated finite reflection group is the symmetric group $\mathcal{S}_n$ on $n$-letters, acting on $\R^n$ by permutation of the coordinates. As there is only one $\mathcal{S}_n$-orbit in $A_{n-1}$, a multiplicity $k\ge 0$ consists of a single parameter. \\
Dunkl theory of type $A$ is closely related to radial analysis on symmetric cones, see for instance \cite{BF98,R07,R20,BR23}. In his unpublished manuscript \cite{M13} from 1987/88, Macdonald studied generalizations of results in radial analysis on symmetric cones, where spherical polynomials of the symmetric cone are replaced by Jack polynomials with arbitrary index, and by introducing a suitable Laplace transform. In contrast to Macdonald, the authors in \cite{BF98} studied the case of non-symmetric Jack polynomials and pointed out the connection to type $A$ Dunkl theory. For the precise connection to the theory of symmetric cones, the reader is referred to \cite{R07,R20,BR23}. \\
The Riesz distribution $R_\mu \in \mathscr{S}'(\R^n)$ for $\mu> \mu_0 \coloneqq k(n-1)$, associated to $(A_{n-1},k)$, is defined as the positive tempered distribution
$$\braket{R_\mu,f}\coloneqq \frac{1}{\Gamma_n(\mu)}\int_{\R_+^n} f(x)\Delta(x)^{\mu-\mu_0-1}\omega(x)\dif x,$$
where $\R_+^n\coloneqq (0,\infty)^n$, $\Delta(x)=x_1\cdots x_n$ and $\Gamma_n(\mu)$ is the generalized gamma function
$$\Gamma_n(\mu)\coloneqq \frac{c_k}{(2\pi)^{\frac{n}{2}}}\prod\limits_{j=1}^n \Gamma(\mu-k(j-1)).$$
The following results were proven in \cite{R20}:
\begin{enumerate}
\item[\rm (i)] $\mu \mapsto R_\mu$ extends to a (weak) holomorphic map $\C \to \mathscr{S}'(\R^n)$.
\item[\rm (ii)] $\Delta(T)R_\mu = R_{\mu-1}$.
\item[\rm (iii)] $\Delta\cdot R_\mu = \prod_{j=1}^n (\mu-k(j-1)) \cdot R_{\mu+1}$.
\item[\rm (iv)] $\supp \, R_\mu \tm \overline{\R_+^n}$.
\item[\rm (v)] $R_0=\delta_0$.
\item[\rm (vi)] $R_\mu$ is a positive measure iff and only if $\mu$ is contained in the generalized Wallach set
$$W_k\coloneqq \set{0,k,\ldots,k(n-1)} \cup (k(n-1),\infty).$$
\end{enumerate}
Riesz distributions on a symmetric cone define a group of tempered distributions under convolution, which is still an open question for Dunkl type Riesz distributions. Indeed, it remained open so far whether two Riesz distributions can be convolved. We shall prove the following theorem.

\begin{theorem}
For $\mu,\nu \in \C$, the Riesz distributions $R_\mu, R_\nu$ are $\mathcal{S}_n$-convolvable and
$$R_\mu*_kR_\nu =R_{\mu+\nu}.$$
\end{theorem}

\begin{definition}
For $s \in \R$ we define $M_s(\R_+^n)$ as the space of complex Radon measures $\mu$ on $\overline{\R_+^n}$ such that $x\mapsto e^{-\braket{\underline{s},x}}$ is integrable with respect to the total variation $\abs{\mu}$, where $\underline{s}=(s,\ldots,s)$. The Dunkl-Laplace transform on $M_s(\R_+^n)$ is defined as 
$$\mathcal{L}\mu(z) \coloneqq \int_{\R^n_+} E(-z,x) \; \mathrm{d}\mu(x), \quad z \in \C^n, \, \mathrm{Re}\, z>s,$$
where $\Re\, z>s$ is understood componentwise. We recall that by \cite[Lemma 3.1]{R20}, the Dunkl kernel satisfies for $z \in \C^n$ with $\Re\, z>s$
$$|E(-x,z)| \le e^{-\braket{\underline{s},x}}, \quad x \in \R_+^n.$$
Therefore, $\mathcal{L}\mu$ is a holomorphic function on $\set{\mathrm{Re}\, z>s}$.
\end{definition}

Furthermore, in \cite{R20}, a Dunkl-Laplace transform for tempered distributions $u$ with support contained in $\overline{\R_+^n}$ was defined by 
$$\mathcal{L}u(z)\coloneqq \braket{u,\tilde{E}(-z,\m)}, \quad \Re\, z>0,$$
where $\tilde{E}(-z,\cdot) \in \mathcal{S}(\R^n)$ and $\tilde{E}(-z,x)=E(-z,x)$ for $x \in \R_+^n$. In fact, $\mathcal{L}u$ is holomorphic, does not depend on the extension $\tilde{E}$ of the Dunkl kernel and satisfies
\begin{equation}\label{LaplaceTranslation}
\mathcal{L}(e^{-\braket{\underline{s},\m}}u)(z)=\mathcal{L}u(\underline{s}+z), \quad s>0.
\end{equation}
Moreover, the Dunkl-Laplace transform is injective in the following sense. If $\mathcal{L}u(\underline{s}+iy)=0$ for some $s>0$ and all $y \in \R^n$, then $u=0$.

\begin{lemma}\label{LaplaceMeasure}
Let $\mu,\nu \in M_s(\R_+^n)$. Then
\begin{enumerate}
\item[\rm (i)] $\mu*_k\nu$ exists and $e^{-\braket{\underline{s},\cdot}}(\mu*_k\nu) \in \mathscr{S}'(\R^n)$ with support contained in $\overline{\R_+^n}$.
\item[\rm (ii)] As holomorphic functions on $\set{\Re \, z>0}$.
$$\mathcal{L}(e^{-\braket{\underline{s},\cdot}}(\mu*_k\nu))(z)=\mathcal{L}\mu(\underline{s}+z) \cdot \mathcal{L}\nu(\underline{s}+z).$$
\end{enumerate}
\end{lemma}

\begin{proof}
\
\begin{enumerate}
\item Since the Dunkl transform is continuous on the Schwartz space, and \\
$|E(ix,y)|\le 1$ for $x,y \in \R^n$, we obtain
\begin{align*}
|\tau_xf(y)| &\le \frac{1}{c_k}||\widehat{f}^k||_{L^1(\R^n,\omega)} \le \tilde{C}||\braket{x}^{\tilde{N}}\widehat{f}||_{\infty} 
\le C||\braket{x}^{N}\partial^\alpha f||_{\infty},
\end{align*}
for some constants $C,\tilde{C}>0,$ $N,\tilde{N} \in \N$ and $\alpha \in \N_0^n$, all independent of $f$ and $x,y$. Moreover, since $\R\underline{1}=A_{n-1}^\perp$, we use \eqref{TranslationFreeDirection} to see that
$$\tau_x(e^{-\braket{\underline{s},\cdot}}f)(y)=e^{-\braket{\underline{s},x+y}} \cdot \tau_xf(y)$$
holds for all $s \in \R$, $f \in C^\infty(\R^n)$ and $x,y \in \R^n$. Since $\overline{\R_+^n}$ is a proper $\mathcal{S}_n$-invariant closed convex cone, Corollary \ref{DistConeAlg} shows that $\mu,\nu$ are $\mathcal{S}_n$-convolvable with  $\mathrm{supp}(\mu*_k \nu) \tm \overline{\R_+^n}$. Finally, we observe that
\begin{align*}
\braket{e^{-\braket{\underline{s},\cdot}}(\mu*_k\nu),\varphi} &= \braket{\mu \otimes \nu, \tau(e^{-\braket{\underline{s},\m}}\varphi)} \\
&= \int_{\overline{\R_+^n}}\int_{\overline{\R_+^n}} e^{-\braket{\underline{s},x+y}}(\tau_xf)(y) \; \mathrm{d}\mu(x)\mathrm{d}\nu(y),
\end{align*}
so that $e^{-\braket{\underline{s},\cdot}}(\mu*_k\nu) \in \mathscr{S}'(\R^n)$. 
\item This is an immediate consequence of $(\tau_x E(-z,\cdot))(y) = E(-z,x)E(-z,y)$ together with equation \eqref{LaplaceTranslation}.
\end{enumerate}
\end{proof}

\begin{proof}[Proof of Theorem 7.1]
By Theorem \ref{PropertiesDunklConvolution} and $\Delta(T)R_\mu=R_{\mu-1}$, we can assume without loss of generality, that $\mathrm{Re}\, \mu, \mathrm{Re}\, \nu > \mu_0$. In this case, $R_\mu,R_\nu \in M_s(\R_+^n)$ for all $s>0$. But, due to \cite[Theorem 5.9]{R20} we have
$$\mathcal{L}R_\alpha=\Delta^{-\alpha}.$$
Together with Lemma \ref{LaplaceMeasure}, this leads to
$$\mathcal{L}(e^{-\braket{\underline{s},\cdot}}(R_\mu*_kR_\nu))(z)=\Delta(\underline{s}+z)^\mu\Delta(\underline{s}+z)^\nu=\Delta(\underline{s}+z)^{\mu+\nu}=\mathcal{L}(e^{-\braket{\underline{s},\cdot}}R_{\mu+\nu})(z),$$
for all $s>0$ and $z \in \C^n$ with $\mathrm{Re}\, z>0$. Finally, the injectivity of the Dunkl-Laplace transform finishes the proof.
\end{proof}

\section*{Acknowledgments}
The author was supported by the Deutsche Forschungsgemeinschaft [RO 1264/4-1].

\end{document}